\documentclass[11pt]{amsart}

\usepackage[utf8x]{inputenc}
\usepackage[english]{babel}
\usepackage[all,cmtip]{xy}

\usepackage[pdftex]{graphicx}
\usepackage[margin=2.8cm]{geometry}

\usepackage{amsfonts}
\usepackage{amsmath}
\usepackage{amssymb}
\usepackage{amsthm}
\usepackage{dsfont}
\usepackage{tikz}
\usepackage{pgfplots}
\usepackage{MnSymbol}
\usepackage{tikz}

\usepackage[T1]{fontenc}

\usepackage{paralist}

\usepackage[colorlinks,link color=blue, cite color=blue,pagebackref,pdftex]{hyperref}

\renewcommand\emptyset{\varnothing}
\renewcommand\phi{\varphi}
\newcommand\N{\mathbb{N}}               
\newcommand\Z{\mathbb{Z}}

\newcommand\R{\mathbb{R}}

\newcommand\dvol{\mathrm{E}}

\newcommand\hri[2]{^{\langle #1,#2\rangle}}
\newcommand\hr[1]{^{\langle #1\rangle}}

\DeclareMathOperator*{\relint}{relint}

\DeclareMathOperator*{\conv}{conv}

\newtheorem{thm}{Theorem}[section]
\newtheorem*{thm*}{Theorem}
\newtheorem{cor}[thm]{Corollary}
\newtheorem{lem}[thm]{Lemma}
\newtheorem{prop}[thm]{Proposition}
\newtheorem{conj}{Conjecture}
\newtheorem{quest}{Question}

\theoremstyle{definition}

\newtheorem{example}[thm]{Example}

\title[Real-rootedness of the Veronese construction]{On the 
real-rootedness of the Veronese construction for 
rational formal power 
series}

\author{Katharina Jochemko}
\address{Institut für Diskrete Mathematik und Geometrie, %
Technische Universit\"at Wien, %
Austria}
\email{katharina.jochemko@tuwien.ac.at}

\keywords{Veronese submodule, Hilbert series, Ehrhart h*-vector, combinatorial 
positive valuations, interlacing polynomials}
\subjclass[2010]{05A10, 05A15, 13A02, 26C10, 52B20, 52B45}

\date{\today}

\begin{document}

\maketitle

\begin{abstract}
We study real sequences $\lbrace 
a_{n}\rbrace _{n\in \N}$ that eventually agree with a polynomial. We 
show that if the numerator polynomial of its rational generating series is of degree $s$ and has 
only nonnegative coefficients, then the numerator polynomial of the subsequence 
$\lbrace a_{rn+i}\rbrace _{n\in \N}$, $0\leq i<r$, has only nonpositive, real roots 
for all $r\geq s-i$. We apply our results to combinatorially positive valuations on polytopes and to Hilbert 
functions of Veronese submodules of graded Cohen-Macaulay algebras. In 
particular, we prove that the Ehrhart $h^\ast$-polynomial of the $r$-th dilate 
of a $d$-dimensional polytope has only distinct, negative, real roots if 
$r\geq \min \{s+1,d\}$. This proves a conjecture of Beck and Stapledon (2010). 
\end{abstract}


\section{Introduction} \label{sec:intro}
We study real sequences $\lbrace a_{n}\rbrace _{n\in \N}$ that eventually 
agree with a polynomial. That is, there is a polynomial $f$ of some degree 
$d-1\geq 0$ in $n$ 
and 
a natural number $n_0$ such that $f(n)=a_n$ for all $n\geq n_0$. In the 
language of generating functions this is equivalent to
\[
\sum _{n\geq 0}a_nt^n \ = \ \frac{h(t)}{(1-t)^{d}},
\]
where $h(t)\in \R[t]$ is a polynomial with $h(1)\neq 0$. In this case, also the 
subsequence 
$\lbrace 
a_{rn+i}\rbrace _{n\in \N}$ eventually agrees with a polynomial for all $0\leq i<r$. We investigate the numerator polynomial $U_{r,i}^d h(t)$ of
\[
\sum _{n\geq 0}a_{rn+i}t^n \ = \ \frac{U_{r,i}^d h(t)}{(1-t)^{d}}
\]
and its behavior as a function of $r$. We will focus 
on the case when $h(t)$ has nonnegative coefficients. Rational formal power series of that form are ubiquitous in combinatorics
and commutative algebra. Our main motivation comes from Ehrhart theory and, more generally, valuations on polytopes. Ehrhart~\cite{ehrhartRational} showed that the number of lattice points in the $n$-th dilate of a $d$-dimensional lattice polytope $P\subseteq \R^m$ is given by a polynomial $E_P(n):=|nP\cap \Z^m|$ in $n$ of degree $d$ for all integers $n\geq 0$. The Ehrhart series of $P$ is 
    \[
    \sum _{n\geq 0}E_P(n)t^n = \frac{h^\ast (P)(t)}{(1-t)^{d+1}},
    \]
    where $h^\ast(P)\in \Z[t]$ is the so-called Ehrhart $h^\ast$-polynomial of 
    $P$ with $\deg h^\ast\leq d$ and $h^\ast(P)(0)=1$. Studying properties of $h^\ast$-polynomials of lattice polytopes is an active area of research. A fundamental Theorem 
    is due to Stanley~\cite{Stanley78} who proved that the coefficients of the 
    $h^\ast$-polynomial are nonnegative integers for every lattice polytope. 
    The function counting lattice points in lattice polytopes is an example of 
    a translation-invariant valuation. McMullen~\cite{mcmullenEuler} 
    generalized Ehrhart's result and showed that for every 
    translation-invariant valuation $\phi$ the function $\phi(nP)$ is given by 
    a polynomial of degree at most $d$ for integers $n\geq 0$. This allows for 
    considering $h^\ast$-polynomials with respect to arbitrary 
    translation-invariant valuations. In \cite{jochemko2015combinatorial} 
    together with Sanyal we characterized all valuations such that the 
    coefficients of the $h^\ast$-polynomials are nonnegative and called them combinatorially positive. 
    Further examples are the volume and the solid-angle sum. For the latter this has been proved by 
    Beck, Robins and Sam \cite{Becksolidangles}. We observe that the operator 
    $h^\ast(P) \mapsto U_{r,i}^dh^\ast(P)$ is related to dilating the polytope $P$. For $i=0$ this operator was considered by Beck and Stapledon \cite{Beck2010} and indeed $U_{r,0}^dh^\ast(P)$ is the $h^\ast$-polynomial of $rP$. For $i\neq 0$ our 
    motivation comes from commutative algebra. For a graded $k$-algebra 
    $A=\bigoplus_{n\geq 0}A_n$ 
    over a field $k$ the Hilbert function is defined by $H_A(n)=\dim_k A_n$. By 
    a 
    theorem of Hilbert (see, e.g.\ \cite[Section 10.4]{eisenbud2013commutative}) this function eventually becomes polynomial if $A_0=k$ 
    and 
    $A$ is generated by finitely many elements in $A_1$. The 
    Hilbert series of $A$ is  
    \[
    \sum _{n\geq 0}H_A(n)t^n=\frac{h(t)}{(1-t)^d}
    \]
    for some polynomial $h(t)\in \Z[t]$ with $h(0)=1$ and $d$ is the Krull 
    dimension of $A$. The 
    $r$-th 
    Veronese subalgebra is defined by $A\hr{r}=\bigoplus_{n\geq 0}A_{rn}$, and 
    more 
    general, for $0\leq i<r$, 
    $A\hri{r}{i}:=\bigoplus_{n\geq 
        0}A_{rn+i}$ is a graded submodule over $A\hr{r}$ called Veronese 
    submodule. In particular, $H_A(rn+i)=H_{A\hri{r}{i}}(n)$ is the Hilbert function of the Veronese 
    submodule. The algebraic behavior of $A\hr{r}$ for large $r$ has been 
    studied by 
    Backelin~\cite{backelin1986rates}, Eisenbud, Reeves and Totaro 
    \cite{eisenbud1994initial} and the submodules $A\hri{r}{i}$ for arbitrary shifts $i<r$ were studied by Aramova, Barca-Nescu und 
    Herzog~\cite{aramova1995rate}. We consider  
    the numerator polynomial $U_{r,i}^d h(t)$ for algebras, for which $h(t)$ 
    has nonnegative coefficients. This is for example 
    the case, when $A$ is a Cohen-Macaulay algebra (see, e.g.\ \cite[Section 4.4]{winfried1998cohen}).
    
    A sequence $a_0,a_1,\ldots, a_m$ of positive integers is called unimodal 
    if  $a_0\leq \cdots \leq a_k \geq \cdots \geq a_m$ for some $0\leq k\leq 
    m$. It is called log-concave if $a_k^2\geq a_{k-1}a_{k+1}$ for all $0<k<m$. 
    If all inequalities are strict, then we obtain strict log-concavity. It is not hard to see that strict 
    log-concavity implies unimodality. An even stronger property is, 
    that the polynomial $a_0+a_1t+\cdots + a_mt^m$ is real-rooted. Brenti and Welker \cite{BrentiWelker} showed that for 
    every polynomial $h(t)$ with $h(0)=1$ the polynomial $U_{r,0}^dh(t)$ has 
    only distinct, negative, real zeros whenever $r$ is sufficiently large. In 
    particular, the coefficients of $U_{r,0}^dh(t)$ are positive and form a 
    strict log-concave sequence. Beck and Stapledon~\cite{Beck2010} 
    strengthened this result by showing that there is an $R>0$ which only 
    depends on $d$ such that for all $r>R$ the polynomial $U_{r,0}^dh(t)$ has 
    only real roots. The optimal bound $R$ was hitherto unknown. In case of 
    Ehrhart polynomials Beck and Stapledon \cite{Beck2010} conjectured the 
    following.
\begin{conj}[{\cite[Conjecture 5.1]{Beck2010}}]\label{conj:Beck}
For a $d$-dimensional lattice polytope $P$ the $h^\ast$-poly\-nomial of $rP$ has only distinct, negative, real zeros for all $r\geq d$.
\end{conj}
In support of this conjecture Higashitani \cite{higashitani2014unimodality} showed that the $h^\ast$-polynomial of $rP$ has strictly log-concave coefficients for $r\geq \deg h^\ast (P)$. We settle this conjecture by proving the following
\begin{thm}\label{thm:maintheorem}
Let $\left\{a_n\right\}_{n\geq 0}$ be a sequence of real numbers such that
    \[
\sum _{n\geq 0}a_nt^n=\frac{h(t)}{(1-t)^d}
    \]
for some integer $d$ and some polynomial $h\neq 0$ of degree $s$ with 
nonnegative coefficients. For all $0\leq i<r$ let 
     \[
     \sum _{n\geq 0}a_{rn+i}t^n=\frac{U_{r,i}^dh(t)}{(1-t)^d}.
     \]
Then $U_{r,i}^dh(t)$ has only nonpositive, real roots for all $r\geq s-i$, and all these 
roots are negative if and only if $h_j>0$ for some $0\leq j\leq i$.
 
 Moreover, 
 \begin{enumerate}
     \item[(i)] if $r>s$, or
     \item[(ii)] if $r=s$ and $s=1$ or $h_i>0$ for some $0<i<s$
    \end{enumerate}
 then all roots are distinct.
\end{thm}
For the proof we employ an interlacing familily of polynomials. 
Interlacing polynomials already turned out to be a key ingredient for various 
other problems (see e.g.\ 
\cite{KadisonSinger,savage2015s,chudnovsky2007roots}). In Section~
\ref{section:interlacing} we introduce all necessary preliminaries on interlacing 
families. That section is self-contained. For further reading on real-rooted 
polynomials we recommend the book by Fisk~\cite{fisk2006polynomials}. 
Section~\ref{section:powerseries} is devoted to the proof of 
Theorem~\ref{thm:maintheorem}. There we construct a family of polynomials $\lbrace 
a\hri{r}{i}_d \rbrace$ that mutually interlace. 
Expressing $U_{r,i}^dh(t)$ in terms of these polynomials then will yield a 
proof of Theorem \ref{thm:maintheorem}. Considering the limiting behavior of $\lbrace 
a\hri{r}{i}_d \rbrace$ we reprove that the roots of $U_{r,i}^dh(t)$ converge to the roots of the Eulerian polynomial. This was first shown by Beck and Stapledon \cite{Beck2010} and Brenti and Welker \cite{BrentiWelker}. In Section~\ref{section:applications} 
we apply our results to combinatorially positive valuations. In particular we 
prove Conjecture~\ref{conj:Beck}. We furthermore point out implications 
for Hilbert series of standard graded algebras. In Section~\ref{sec:optimality} we prove that the bounds given in Theorem \ref{thm:maintheorem} are in fact optimal. We conclude by considering optimality for Ehrhart series in Section \ref{sec:openquest}.

\section{Interlacing polynomials}\label{section:interlacing}
Let $f,g \in \R[t]$ be non-zero polynomials with only real roots, let $\deg f=l$ and $\deg g=m$. Let $t_1\geq t_2 \geq \cdots  \geq t_l$ be the roots of $f$ and $s_1\geq s_2 \geq \cdots  \geq s_m$ be the roots of $g$. Then $g$ 
interlaces $f$ if either $l=m$ and
\[
s_m \ \leq \ t_m \ \leq  \ldots \ \leq \ s_1 \ \leq \ t_1 \ ,
\]
or $l=m+1$ and 
\[
t_{m+1} \ \leq \ s_m \ \leq \ t_m \ \leq  \ldots \ \leq \ s_1 \ \leq \ t_1 \
\]
and we write $g\preceq f$. If all inequalities are strict, then $g$ 
strictly interlaces $f$ and we write $g\prec f$. In particular, $f$ and $g$ have only distinct, real roots if $g\prec f$.

The following lemma collects some well-known facts about interlacing polynomials. (For further reading see e.g.\ \cite{fisk2006polynomials}.)
\begin{lem}\label{lem:addrealroots}
    Let $f,g,h \in \R[t]$ be non-zero polynomials with only real roots 
    and positive 
    leading coefficients. Then
    \begin{enumerate}
        \item[(i)] if $f,g$ have only negative roots, then $g \prec f$ if and only if $f \prec tg$.
        \item[(ii)] $g \prec f$ if and only if $cg \prec df$ for all $c,d>0$.
        \item[(iii)] if $h\prec f$ and $h\preceq g$ then $h\prec f+g$.
        \item[(iv)] if $f\prec h$ and $g\preceq h$ then $f+g \prec h$.
    \end{enumerate}
Furthermore, statements (i) -- (iv) remain true if $\prec$ is replaced by $\preceq$.
\end{lem}
\begin{proof}
In order to keep this section self-contained, we give a proof here. Statements 
(i) and (ii) are trivial. For (iii) let $u_1> u_2 > 
\cdots  > u_m$ be the roots of $h$. Since $h$ interlaces $f$ and $g$ we have 
$\deg f,\deg g \in \lbrace m, m+1 \rbrace$. As an example, we prove the case 
$\deg f=m+1$ and $\deg g=m$, and all other cases follow in a very similar way. 
Since $h$ strictly interlaces $f$, $f(u_i)<0$ if $i$ is odd and $f(u_i)>0$ if $i$ is even, and the same for $g$ if we replace all strict inequalities by weak inequalites. Therefore we obtain
\[
(f+g)(u_i)\begin{cases}
<0 & \text{ if } i \text{ is odd, }\\
>0 & \text{ if } i \text{ is even. }
\end{cases}
\]
Since $f$ and $g$ have positive leading coefficients, $\lim _{t\rightarrow 
+\infty}(f+g)(t)=\infty$, and $\lim 
_{t\rightarrow -\infty}(f+g)(t)=-\infty$ if $m$ is even and $\lim 
_{t\rightarrow -\infty}(f+g)(t)=+\infty$ if $m$ is odd. 
In both cases, by the intermediate value theorem, $f+g$ has at least one zero 
in each of the $m+1$ many open intervals $$(-\infty, u_m)\ , \ 
(u_m,u_{m-1}) \ , \ldots, 
(u_2,u_1)\ , \ (u_1,+\infty).$$ Since $\deg (f+g)=m+1$, the claim follows. If 
$h\preceq f$ we employ a limiting argument: For all $n\geq 0$ let  
$f^{(n)},g^{(n)},h^{(n)}\in \R[t]$ such that $h^{(n)}\prec f^{(n)}$ and 
$h^{(n)}\prec g^{(n)}$ and $\lim _{n\rightarrow \infty}f^{(n)}=f$, $\lim 
_{n\rightarrow \infty}g^{(n)}=g$ and $\lim _{n\rightarrow \infty}h^{(n)}=h$. 
Then $h^{(n)}\prec f^{(n)}+g^{(n)}$ which in the limit yields $h\preceq f+g$. 
Part (iv) follows analogously.
\end{proof}
Let $f_1,\ldots,f_m \in \R[t]$ be an (ordered) collection of polynomials. Then 
$f_1\preceq \cdots \preceq f_m \in \R[t]$ 
mutually interlace if $f_i \preceq f_j$ whenever $i<j$. If $f_i \prec f_j$ for 
$i<j$ then $f_1,\ldots,f_m \in \R[t]$ strictly mutually interlace.

The following proposition is a special case of Fisk \cite[Proposition 
3.72]{fisk2006polynomials}. (Compare also Savage and Visontai \cite[Theorem 
2.3]{savage2015s}.)
\begin{prop}\label{prop:recursion}
Let $f_{r-1} \prec \cdots \prec f_0$ be strictly mutually interlacing polynomials with 
only negative roots and positive leading coefficients. For 
all $0\leq l \leq r-1$ let
\[
g_l=f_0+\cdots +f_l +tf_{l+1}+\cdots +tf_{r-1}.
\]
Then also $g_{r-1} \prec \cdots \prec g_0$ are strictly mutually interlacing, have only 
negative roots and positive leading coefficients. In particular, $g_0,\ldots, g_{r-1}$ have only distinct, negative, real roots.
\end{prop}
\begin{proof}
    For the sake of completeness we repeat Fisk's proof here adapted to our situation. For 
    $0\leq l <m \leq r-1$ 
    we define
    \begin{align*}
    p_1 &= f_0+\cdots +f_l ,\\
    p_2 &= f_{l+1} + \cdots + f_m ,\\
    p_3 &= f_{m+1} + \cdots + f_{r-1}.
    \end{align*}
Then $g_m = p_1+p_2+tp_3$ and $g_l=p_1+tp_2+tp_3$. It is clear, that $g_m$ and 
$g_l$ have positive leading coefficient, and their real roots are negative, 
since the same is true for $f_0,\ldots, f_{r-1}$.

We observe that $p_1,p_2\neq 0$. If $p_3\neq 0$, then by iterative application of Lemma \ref{lem:addrealroots}(iii) and (iv) we obtain $p_3\prec p_2 \prec p_1$. 
Furthermore, since $p_1,p_2$ and $p_3$ have only negative roots, $tp_3\preceq 
tp_2$ and, by Lemma \ref{lem:addrealroots}(i), $p_1\prec tp_2$. Thus, 
by Lemma \ref{lem:addrealroots}(iv), $p_1+tp_3 \prec tp_2$ and since 
$p_1+tp_3\preceq p_1+tp_3$ this yields
\begin{equation}\label{eq:1}
p_1+tp_3 \prec p_1+tp_2+tp_3.
\end{equation} 
Further, since $p_2\prec p_1$, $p_2\preceq tp_2$ and $p_2\prec tp_3$ we have by 
Lemma \ref{lem:addrealroots}(iii)
\begin{equation}\label{eq:2}
p_2 \prec p_1+tp_2+tp_3.
\end{equation}
Applying \ref{lem:addrealroots}(iv) to the equations \eqref{eq:1} and 
\eqref{eq:2} 
completes the proof. If $p_3=0$ then the argument just works out the same way.
\end{proof}

\section{Rational power series}\label{section:powerseries}
For every formal power series $f\in \R\lsem t\rsem $ and all 
integers $r\geq 1$ there are uniquely 
determined $f_0,\ldots, f_{r-1}\in \R\lsem t\rsem $ such that
\[
f(t)=f_0(t^r)+tf_1(t^r)+\cdots + t^{r-1}f_{r-1}(t^r). 
\]
For $0\leq i\leq r-1$ let $\hri{r}{i}\colon \R\lsem t\rsem\rightarrow \R\lsem 
t\rsem$ 
be the linear 
operator defined by
\[
f\hri{r}{i}=f_i.
\]
Our main objects under consideration are the polynomials defined by
\[
a\hri{r}{i}_d(t):=\left((1+t+\cdots + t^{r-1})^d\right)\hri{r}{i}
\]
for all $r\geq 1$ and all $0\leq i\leq r-1$.
Furthermore, for all $0\leq i\leq r-1$ and $d\geq 0$ let $U_{r,i}^d \colon 
\R[t]\rightarrow \R[t]$ be the linear operator on polynomials, such that for all polynomials $h\in \R[t]$
\[
\sum _{n\geq 0}a_{rn+i}t^n=\frac{U_{r,i}^dh(t)}{(1-t)^{d}} \quad \text{ whenever }\quad \sum _{n\geq 0}a_nt^n=\frac{h(t)}{(1-t)^{d}}.
\]

The following lemma clarifies the relation between the operators $U_{r,i}^d$ and $\hri{r}{i}$ and is a slight generalization of \cite[Lemma 
3.2]{Beck2010}.
\begin{lem}\label{lem:hecke}
For every polynomial $h\in \R[t]$ and integers $r\geq 1$, $d\geq 0$ and $0\leq i\leq r-1$
    \begin{equation*}
        U_{r,i}^dh(t)=\left(h(t)(1+t+\cdots + 
        t^{r-1})^d\right)\hri{r}{i}.
    \end{equation*}
    \begin{proof}
        We repeat the argument given in \cite{Beck2010}:
        \begin{align*}
        (1-t)^d\sum_{n\geq 0}a_{rn+i}t^n&=\left((1-t^r)^d\sum _{n\geq 0}a_n 
        t^n\right)\hri{r}{i}\\
        &=\left(h(t)(1+\cdots +t^{r-1})^d\right) \hri{r}{i}.
        \end{align*}
    \end{proof}    
\end{lem}
The next lemma expresses $U_{r,i}^dh(t)$ in terms of 
$a\hri{r}{i}_d$.
\begin{lem}\label{lem:linearexpr}
Let $h\in \R[t]$ be a polynomial. Then 
\begin{equation}\label{eq:operatorcoefficients0}
           U_{r,i}^dh(t)=  h\hri{r}{0}a\hri{r}{i}_d+\cdots 
           + h\hri{r}{i} a\hri{r}{0}_{d}+ t 
        \left(h\hri{r}{i+1}a\hri{r}{r-1}_d +\cdots + h\hri{r}{r-1} 
        a\hri{r}{i+1}_d\right)
\end{equation}
for all integers $d\geq 0$ and $0\leq i\leq r-1$.
\end{lem}
\begin{proof}
    By Lemma \ref{lem:hecke}, $U_{r,i}^dh(t)$ is the coefficient of $t^i$ in
    \begin{align*}
    h(t)\left(1+\cdots +t^{r-1}\right)^d &=\left(h\hri{r}{0}(t^r)+\cdots 
    +h\hri{r}{r-1}(t^r)t^{r-1}\right)\left(a_d\hri{r}{0}(t^r)+\cdots + 
    a_d\hri{r}{r-1}(t^r)t^{r-1}\right)\\
    &=\sum _{i=0}^{r-1}\left( t^i\sum 
    _{k+l=i}h\hri{r}{k}(t^r)a_d\hri{r}{l}(t^r) + 
    t^{i+r}\sum _{k+l=i+r}h\hri{r}{k}(t^r)a_d\hri{r}{l}(t^r)\right)\\
    &=\sum _{i=0}^{r-1}t^i\left(\sum 
    _{k+l=i}h\hri{r}{k}(t^r)a_d\hri{r}{l}(t^r) + 
    \sum _{k+l=i+r}h\hri{r}{k}(t^r)(a_d\hri{r}{l}t)(t^r)\right),
    \end{align*}
    where the sums are taken over all $0\leq k,l <r$ such that $k+l=i$, 
    respectively, $k+l=r+i$.
\end{proof}
If $r$ is large enough, then the coefficients $h\hri{r}{j}$ in Lemma \ref{lem:linearexpr} are linear polynomials, and the expression of $U_{r,i}^dh(t)$ in terms of $a\hri{r}{j}_d$ simplifies.
\begin{lem}\label{lem:lincomb}
Let $h=h_0+h_1t+\cdots +h_st^s\in \R[t]$ be a polynomial of degree $s$. Then 
for all integers $d\geq 0$, $0\leq i<r$ and $r\geq s-i$
\begin{equation}\label{eq:operatorcoefficients}
    U_{r,i}^dh(t)=h_0a\hri{r}{i}_d+\cdots + h_i a\hri{r}{0}_{d}+ t 
    \left(h_{i+1}a\hri{r}{r-1}_d +\cdots + h_{r+i} a\hri{r}{0}_d\right),
\end{equation}
where $h_j:=0$ for all $j>s$.
\end{lem}
\begin{proof}
Since $s\leq r+i<2r$, the polynomials $h\hri{r}{j}$, $0\leq j<r$, are at most linear, namely
\[
h\hri{r}{j} \ = \ h_j+h_{r+j}t
\]
for all $0\leq j<r$. Therefore, by Lemma \ref{lem:linearexpr},
\begin{eqnarray*}
U_{r,i}^dh(t)  &=&  (h_0+th_r)a\hri{r}{i}_d+\cdots +(h_i+th_{r+i})a\hri{r}{0}_d \\
     &&  +t \left((h_{i+1}+th_{r+i+1})a\hri{r}{r-1}_d +\cdots + (h_{r-1}+th_{2r-1}) 
        a\hri{r}{i+1}_d\right)\\
        &=&  h_0a\hri{r}{i}_d+\cdots +h_ia\hri{r}{0}_d +  t \left(h_{i+1}a\hri{r}{r-1}_d +\cdots + h_{r+i} 
        a\hri{r}{0}_d\right),
\end{eqnarray*}
where the last equality follows since $h_{r+i+j}=0$ for all $j>0$.
\end{proof}
The following proposition should be compared with \cite[Example 3.76.]{fisk2006polynomials}.
\begin{prop}\label{prop:keyprop}
    For all $r\geq 1$ and $d\geq 1$ the polynomials
    \[
    a_d\hri{r}{r-1}\ \prec \ \cdots \ \prec \ a_d\hri{r}{0}
    \]
    are strictly mutually interlacing, have only negative roots and positive 
    leading coefficients.
\end{prop}
\begin{proof}
    We use induction on $d$. Since by definition $a_1\hri{r}{r-1}= \cdots = 
    a_1\hri{r}{0}\equiv 1$ the statement is trivially true for $d=1$. For $d\rightarrow d+1$ we therefore obtain by Lemma~\ref{lem:linearexpr} and Lemma~\ref{lem:lincomb}
        \[
    a\hri{r}{i}_{d+1} = a_d\hri{r}{0}+\cdots 
        +a_d\hri{r}{i}+ta_d\hri{r}{i+1}+\cdots +ta_d\hri{r}{r-1}.
    \]
    The proof now follows from Proposition \ref{prop:recursion}.
\end{proof}
\begin{proof}[Proof of Theorem \ref{thm:maintheorem}]
By Lemma \ref{lem:lincomb}
\begin{equation}\label{eq:operatorcoefficients}
    U_{r,i}^dh(t)=h_0a\hri{r}{i}_d+\cdots + h_i a\hri{r}{0}_{d}+ t 
    \left(h_{i+1}a\hri{r}{r-1}_d +\cdots + h_{r+i} a\hri{r}{0}_d\right)
\end{equation}
whenever $r\geq s-i$. Since the polynomials $a\hri{r}{i}_d$ have only negative real roots and positive leading coefficients, all  real zeros of $U_{r,i}^dh$ are negative if and only if there is a $0\leq j \leq i$ with $h_j>0$. 

In order to show real-rootedness we distinguish the cases $s\leq i<r$, $i<s<r$ and $i<s= r$. If $s\leq i<r$, then $U_{r,i}^dh =h_0a\hri{r}{i}_d+\cdots + h_s a\hri{r}{i-s}_{d}$ and by applying Proposition \ref{prop:keyprop}, Lemma \ref{lem:addrealroots}(ii) and Proposition~\ref{prop:recursion} to the non-zero summands, we see that $U_{r,i}^dh$ has only negative, distinct, real roots.

    If $i<s<r$ then $U_{r,i}^dh=h_0a\hri{r}{i}_d+\cdots + h_i a\hri{r}{0}_{d}+ 
    t\left(h_{i+1}a\hri{r}{r-1}_d +\cdots + h_s a\hri{r}{r+i-s}_d\right)$. If 
    $h_0a\hri{r}{i}_d+\cdots + h_i a\hri{r}{0}_{d}=0$ then again by the same 
    arguments as before applied to $h_{i+1}a\hri{r}{r-1}_d +\cdots + h_s 
    a\hri{r}{r+i-s}_d$, $U_{r,i}^dh$ has only nonpositive, distinct, real roots. 
In the other case, since
\[
a\hri{r}{r-1}_d\prec \cdots \prec a\hri{r}{r+i-s}_d 
\prec 
a\hri{r}{i}_d \prec \cdots \prec a\hri{r}{0}_{d}
\]
are strictly mutually 
interlacing, we 
can again apply Proposition \ref{prop:keyprop}, Lemma 
\ref{lem:addrealroots}(ii) 
    and Proposition~\ref{prop:recursion} to all non-zero summands of 
    $U_{r,i}^dh$ and obtain that $U_{r,i}^dh$ has only negative distinct real 
    roots.
    
    If $i<s=r$ then $U_{r,i}^dh \ = \ 
    h_0a\hri{s}{i}_d+\cdots + h_i a\hri{s}{0}_{d}+ t 
    \left(h_{i+1}a\hri{s}{s-1}_d +\cdots + h_s a\hri{s}{i}_d\right)$. If 
    $h_0=0$ we argue as in the case $i<s<r$. If $h_0\neq 0$ then we observe 
    that by Proposition \ref{prop:keyprop}
    \begin{align*}\label{eq:interlace}
    a\hri{s}{i}_d &  \preceq  ta\hri{s}{i}_d,\\
    a\hri{s}{j}_d & \prec ta\hri{s}{k}_d \text{ for all } j<k,\\
    ta\hri{s}{j}_d & \preceq ta\hri{s}{k}_d \text{ for all } j\geq k.
    \end{align*}
    Applying Lemma \ref{lem:addrealroots}(ii) and (iv) multiple times yields 
    \[
    U_{r,i}^dh \preceq t 
    \left(h_{i+1}a\hri{s}{s-1}_d +\cdots + h_s a\hri{s}{i}_d\right)
    \]
    and thus $U_{r,i}^dh$ has only negative, real roots. If $s>1$ and there is a 
    $0<j<s$ with $h_j>0$ then even $U_{r,i}^dh \prec t 
    \left(h_{i+1}a\hri{s}{s-1}_d +\cdots + h_s a\hri{s}{i}_d\right)$ and 
    $U_{r,i}^dh$ 
    has distinct roots. The reason is that if $i\geq j$ then $h_ja\hri{s}{i-j}_d \prec 
    h_sta\hri{s}{i}_s$ and otherwise 
    $0\neq h_0a\hri{s}{i}_d \prec h_jta\hri{s}{s+i-j}_d$, and 
    $h_ja\hri{s}{i-j}_d$ or, respectively, $h_jta\hri{s}{s+i-j}_d$ appears as a summand of 
    $U_{r,i}^dh$. If $r=s=1$ then $U_{r,0}^dh=h_0+h_st$ is a linear polynomial 
    and thus the roots are distinct. \end{proof}
        
        \subsection{Limiting behavior}
We finish this section by considering the limiting behavior of the polynomials $a\hri{r}{i}_d$ for large $r$. To that end, recall that for all $k\in \N$ 
\[
\sum _{n\geq 0}n^kt^n \ = \ \frac{A_k(t)}{(1-t)^{k+1}},
\]
where $A_k(t)$ denotes the $k$-th Eulerian polynomial. Eulerian polynomials are well-studied and are known to be symmetric and have only nonpositive roots which are all distinct. The following proposition draws a connection to the results of Brenti and Welker \cite{BrentiWelker} and Beck and Stapledon~\cite{Beck2010}. 
\begin{prop}\label{prop:limitsofa}
For all $d\geq 1$ and all $i\in \N$
\[
\lim_{r\rightarrow \infty} \frac{a\hri{r}{i}_d (t)}{r^{d-1}}\ = \ \lim_{r\rightarrow \infty} t\frac{a\hri{r}{r-i}_d (t)}{r^{d-1}} \ = \  \frac{A_{d-1} (t) }{(d-1)!}.
\]
In particular, the roots of $a\hri{r}{i}_d (t)$ converge to the roots of the Eulerian polynomial $A_{d-1}(t)$ when $r$ goes to infinity.
\end{prop}
\begin{proof}
Let $f(n)={n+d-1 \choose d-1}$. Then $\sum _{n\geq 0} f(n)t^n=\frac{1}{(1-t)^d}$. We observe that $f(rn+i)$ is a polynomial of degree $d-1$ in $rn$ with leading coefficient $\tfrac{1}{(d-1)!}$, say $f(rn+i)=\alpha _{d-1}n^{d-1}r^{d-1}+\cdots + \alpha _0$. Then, by definition,
\begin{align*}
\frac{a\hri{r}{i}_d (t)}{r^{d-1}}\ = \ \frac{U_{r,i}^d 1}{r^{d-1}} \ &= \ \alpha_{d-1}A_{d-1}+ \alpha _{d-2}r^{-1}(1-t)A_{d-2} + \cdots + \alpha _0 r^{-d+1}(1-t)^{d-1}A_0\\
 &\xrightarrow[r \to \infty]{} \frac{A_{d-1}(t)}{(d-1)!}.
\end{align*}
Further, $f(rn+r-i)=f(r(n+1)-i)$ is a polynomial of degree $d-1$ in $r(n+1)$ with leading coefficient $\tfrac{1}{(d-1)!}$, say $f(r(n+1)-i)=\beta _{d-1}(n+1)^{d-1}r^{d-1}+\cdots + \beta _0$. Then
\begin{align*}
\frac{a\hri{r}{r-i}_d (t)}{r^{d-1}}\ &= \ \frac{U_{r,r-i}^d 1}{r^{d-1}} \\ 
&= \ \beta_{d-1}t^{-1}A_{d-1}+\cdots + \beta _{1}r^{-d}(1-t)^{d-2}t^{-1}A_{1} + \beta _0 r^{-d+1}(1-t)^{d-1}A_0\\
&\xrightarrow[r \to \infty]{} \frac{t^{-1}A_{d-1}(t)}{(d-1)!}.
\end{align*}
\end{proof}
Together with Lemma \ref{lem:lincomb} this yields
\begin{cor}[{\cite{Beck2010,BrentiWelker}}]\label{cor:limitsoff}
Let $h(t)\neq 0$ be a polynomial in $\R[t]$. Then for all $d\geq 1$ and all $i\in \N$
\[
\lim_{r\rightarrow \infty} \frac{U_d \hri{r}{i} h}{r^{d-1}} \ = \ h(1)\frac{A_{d-1} (t) }{(d-1)!}.
\]
In particular, if $h(1)\neq 0$ then the roots of $U_d \hri{r}{i} h$ converge to the roots of the Eulerian polynomial $A_{d-1}(t)$ when $r$ goes to infinity.
\end{cor}

\section{Applications}\label{section:applications}

\subsection{Combinatorially positive valuations}\label{subsection:combpos}
In this section, let $\Lambda \subseteq \mathbb{R}^m$ be a discrete additive 
subgroup or a vector subspace over a subfield of $\R$. Let 
$\mathcal{P}(\Lambda)$ denote the collection of convex polytopes with vertices 
in $\Lambda$. A translation-invariant valuation or 
$\Lambda$-valuation is a map $\phi \colon \mathcal{P}(\Lambda) 
\rightarrow G$, where $G$ is an abelian group, such that $\phi(\emptyset)=0$,
\[
\phi (P\cup Q)=\phi(P)+\phi(Q)-\phi(P\cap Q)
\]
whenever $P,Q,P\cup Q,P\cap Q\in \mathcal{P}(\Lambda)$, and
\[
\phi (P+t)=\phi(P)
\]
for all $P\in \mathcal{P}(\Lambda)$ and $t\in \Lambda$. A $\Lambda$-valuation 
is called even if $\phi(P)=\phi(-P)$ for all $P\in 
\mathcal{P}(\Lambda)$. The standard example for $\Lambda=\R^m$ is the volume. An important example for $\Lambda=\Z^m$ is the 
discrete volume $E(P):=|P\cap \Z^m|$. Ehrhart \cite{ehrhartRational} 
showed that the 
lattice point 
enumerator agrees with a polynomial under dilation of the polytope. More 
specifically, for a $d$-dimensional polytope $P\in 
\mathcal{P}(\Lambda)$ the function $E_P(n):=E(nP)$ is given by a polynomial of 
degree $d$ for 
$n\geq 0$. McMullen~\cite{mcmullenEuler} showed more generally that for a 
$\Lambda$-valuation $\phi$ 
and a $d$-dimensional polytope $P\in \mathcal{P}(\Lambda)$ the function 
$\phi_P(n):=\phi(nP)$ agrees with a polynomial of degree at most $d$. In terms 
of generating functions, this can be expressed as
\[
\sum_{n\geq 0}\phi_P(n)t^n \ = \ \frac{h^\phi(P)(t)}{(1-t)^{d+1}},
\]
where $h^\phi(P)(t)=h_0^\phi(P)+\cdots +h_s^\phi (P)t^s$ is a polynomial with 
coefficients in $G$ of degree $s$ for some $s\leq d$. This is equivalent to
\begin{equation}\label{eq:binombasis}
\phi _P (n) \ = \ h^\phi _0 (P){n+d\choose d} \ + \ h^\phi _1 (P){n+d-1\choose  d} + \cdots + h^\phi _s (P){n+d-s\choose  d}.
\end{equation}
McMullen \cite{mcmullenIE} showed, that all $\Lambda$-valuations satisfy an inclusion-exclusion property. Thus, we can define the value of $\phi$ 
on the relative interior of a $d$-dimensional polytope $P\in 
\mathcal{P}(\Lambda)$ as
\[
\phi(\relint P):=\sum _{F}(-1)^{\dim P-\dim F}\phi(F),
\]
where the sum is taken over all faces $F$ of $P$. Therefore, by M\"obius inversion 
\[
\phi (P) \ = \sum _{F}\phi(\relint F).
\]
In the following, let $G$ be an ordered abelian group. We call a $\Lambda$-valuation $\phi$ combinatorially 
positive if $h_0^\phi(P),\ldots, h_s^\phi (P)\geq 0$ for all $P\in 
\mathcal{P}(\Lambda)$. Together with Sanyal, we characterized in 
\cite{jochemko2015combinatorial} all 
combinatorially positive 
$\Lambda$-valuations.
\begin{thm}[{\cite[Theorem 3.6]{jochemko2015combinatorial}}]\label{thm:combpositiveRamanKath}
    Let $\phi \colon  \mathcal{P}(\Lambda)\rightarrow G$ be a 
    $\Lambda$-valuation. Then the following are equivalent:
    \begin{enumerate}
        \item[(i)] $\phi$ is combinatorially positive.
        \item[(ii)] $\phi(\relint \Delta)\geq 0$ for all simplices $\Delta \in 
        \mathcal{P}(\Lambda)$.
    \end{enumerate}
\end{thm}
From the proof of {\cite[Theorem 3.6]{jochemko2015combinatorial}} it furthermore follows that $h_0^\phi(P)=\phi(\lbrace 0\rbrace)$, $h_1^\phi(P)=\phi(P)-(d+1)\phi(\lbrace 0 \rbrace)$ and  $h_d^\phi (P)=\phi(\relint (-P))$. The following Proposition generalizes a well-known inequality for the coefficients of the Ehrhart $h^\ast$-polynomial.
\begin{prop}\label{prop:h_1}
    Let $\phi \colon  \mathcal{P}(\Lambda)\rightarrow G$ be an even and 
    combinatorially 
    positive $\Lambda$-valuation and let $P 
    \in \mathcal{P}(\Lambda)$ be a $d$-dimensional polytope. Then $h_1^\phi 
    (P)\geq h_d^\phi (P).$
\end{prop}
\begin{proof}
Since $\phi$ is even we obtain $h_d^\phi (P)=\phi(\relint (P))$. Let $\mathcal{C}$ be a triangulation of $P$ using only 
vertices of $P$. Then
\begin{align*}
\phi(P) =\sum _{C\in \mathcal{C}}\phi(\relint (C)) & \geq \sum _{C\in 
\mathcal{C}\atop \relint C \subseteq \relint P}\phi(\relint (C))+\sum _{C\in 
\mathcal{C}\atop \dim C=0}\phi(\relint (C)) \\
& \geq \phi(\relint P) + 
(d+1)\phi(\lbrace 0 \rbrace)
\end{align*}
since every polytope of dimension $d$ has at least $d+1$ vertices.
\end{proof}

From Theorem 
    \ref{thm:maintheorem} we obtain the following corollary.
\begin{cor}\label{cor:thmvaluations}
    Let $\phi \colon  \mathcal{P}(\Lambda)\rightarrow \R$ be a combinatorially 
    positive $\Lambda$-valuation and let 
    $P 
    \in \mathcal{P}(\Lambda)$ be a $d$-dimensional polytope. Let 
    $s=\deg h^\phi(P)$. Then $h^\phi(rP)(t)$ has only real roots for all $r\geq s$, and for $r>s$ all these roots are distinct. The roots are negative whenever $\phi (\lbrace 0 \rbrace)>0$. If $\phi$ is even, then all roots are distinct whenever  $r\geq \min \{s+1,d\}$.
\end{cor}
\begin{proof}
It remains to prove the case $r=d=s$ when $\varphi$ is even. Either $s=1$ or $s>1$ and we observe $h_1^\phi(P)\geq 
    h_d^\phi (P)>0$ by Proposition \ref{prop:h_1}. In both cases we conclude by 
    Theorem \ref{thm:maintheorem}. Furthermore, by 
    Theorem \ref{thm:maintheorem} all roots are negative if and only if  $h_0^\phi(P)=\phi(\{0\})>0$.
\end{proof}
Applied to the discrete volume $E\colon \mathcal{P}(\Z^m)\rightarrow 
\Z$ this proves a strengthened version of a conjecture of Beck and Stapledon 
\cite{Beck2010}.
\begin{cor}
    Let $P$ be a $d$-dimensional lattice polytope. Let $s=\deg h^\ast (P)$. 
    Then
    \[
    \sum _{n\geq 0}E_{rP}(n)t^n=\frac{h^\ast(rP)(t)}{(1-t)^{d+1}},
    \]
    where $h^\ast(rP)(t)$ has only negative, real roots for all $r\geq s$ and all roots are distinct if $r\geq \min \{s+1,d\}$.
\end{cor}

The following example shows, that the restriction to even combinatorially positive valuations in Corollary \ref{cor:thmvaluations} and Proposition \ref{prop:h_1} is in general necessary.

\begin{example}
 For every translation vector $v\in \R^2$ the map
\[
\dvol ^ v \colon P \mapsto \dvol (P-v)=|(P-v)\cap \Z^2|
\]
defines a $\mathbb{Z}^2$-valuation, which is by Theorem \ref{thm:combpositiveRamanKath} combinatorially positive since $\dvol^v(\relint \Delta) = |\relint (\Delta-v)\cap \Z^2|\geq 0$ for all lattice simplices $\Delta$. Now let $v=(\lambda,\lambda)$ for $0< \lambda <<1$. For the simplex $S=\conv (\{0,-e_1,-e_2\})$ we then have
\[
(-S-v) \cap \Z^2 =\{0\} \text{ and } (S-v) \cap \Z^2 =\{\},
\]
(see Figure \ref{fig:1}) and therefore $h_1^\phi(S)=0\not > h_2^\phi(S)=1$.
\begin{figure}[t]
\begin{tikzpicture}
\filldraw[green!50] (1.95,1.95) -- (1.95,0.95) -- (0.95,1.95) -- (1.95,1.95);
\filldraw [gray] (0,0) circle (1pt);
\filldraw [gray] (1,0) circle (1pt);
\filldraw [gray] (2,0) circle (1pt);
\filldraw [gray] (3,0) circle (1pt);
\filldraw [gray] (0,1) circle (1pt);
\filldraw [gray] (1,1) circle (1pt);
\filldraw [gray] (2,1) circle (1pt);
\filldraw [gray] (3,1) circle (1pt);
\filldraw [gray] (0,2) circle (1pt);
\filldraw [gray] (1,2) circle (1pt);
\filldraw [gray] (2,2) circle (1pt) node[anchor=west] {0};
\filldraw [gray] (3,2) circle (1pt);
\filldraw [gray] (0,3) circle (1pt);
\filldraw [gray] (1,3) circle (1pt);
\filldraw [gray] (2,3) circle (1pt);
\filldraw [gray] (3,3) circle (1pt);
\node at (1.5,1.5) {$S-v$};
\end{tikzpicture}
\hspace{2cm}
\begin{tikzpicture}
\filldraw[green!50] (0.95,0.95) -- (1.95,0.95) -- (0.95,1.95) -- (0.95,0.95);
\filldraw [gray] (0,0) circle (1pt);
\filldraw [gray] (1,0) circle (1pt);
\filldraw [gray] (2,0) circle (1pt);
\filldraw [gray] (3,0) circle (1pt);
\filldraw [gray] (0,1) circle (1pt);
\filldraw [gray] (1,1) circle (1pt) node[anchor=east] {0};
\filldraw [gray] (2,1) circle (1pt);
\filldraw [gray] (3,1) circle (1pt);
\filldraw [gray] (0,2) circle (1pt);
\filldraw [gray] (1,2) circle (1pt);
\filldraw [gray] (2,2) circle (1pt);
\filldraw [gray] (3,2) circle (1pt);
\filldraw [gray] (0,3) circle (1pt);
\filldraw [gray] (1,3) circle (1pt);
\filldraw [gray] (2,3) circle (1pt);
\filldraw [gray] (3,3) circle (1pt);
\node at (1.5,1.5) {$-S-v$};
\end{tikzpicture}
\caption{$S-v$ and $-S-v$.}
\label{fig:1}
\end{figure}
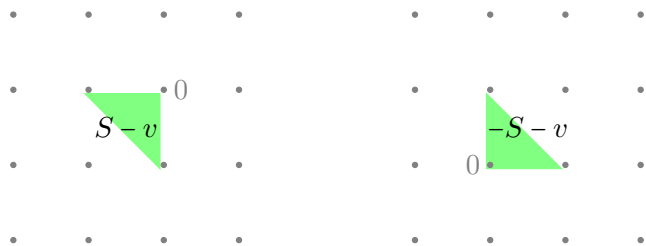 
We obtain
\[
h^{\dvol^v}(S)=t^2.
\]
For the (combinatorially positive) valuation $\tilde{\dvol} = \dvol+\dvol^{v}$ we have
\[
h^{\tilde{\dvol}}(S)=1+t^2.
\]
Lemma \ref{lem:lincomb} gives
\[
U_{2,0}^2 h^{\tilde{\dvol}}(S) = (1+t)a\hri{2}{0}_2 = (1+t)^2
\]
which is real-rooted but with a double root at $t=-1$.
\end{example}

\subsection{Standard graded algebras}\label{subsection:gradedrings}
Let $k$ be a field. A graded $k$-algebra $A=\bigoplus _{n\geq 0}A_n$ is 
standard graded, if $A_0=k$ and it is finitely 
generated in degree $1$. The Hilbert series of $A$ is defined as $H(A,t)=\sum 
_{n\geq 0}\dim _k A_n t^n$. By a Theorem of Hilbert (see e.g.\ \cite[Section 
10.4]{eisenbud2013commutative}) this series is of the form
    \[
    H(A,t)=\frac{h(t)}{(1-t)^d},
    \]
for some polynomial $h$ with $h(0)=1$. Here, $d$ is the Krull dimension of $A$. 
For all $r\geq 1$ and all $0\leq i\leq r-1$ the (shifted) Veronese submodule is 
defined as $A\hri{r}{i}=\bigoplus _{n\geq 0}A_{rn+i}$. Then $A\hri{r}{i}$ is a 
graded $A\hri{r}{0}$-submodule of $A$. The next corollary is a direct consequence of Theorem \ref{thm:maintheorem} and applies to all 
standard graded $k$-algebras $A$ such that the numerator polynomial $h$ of its Hilbert series has only nonnegative coefficients. This is in particular the case if $A$ is a 
Cohen-Macaulay algebra (see e.g.\ \cite[Section 4.4]{winfried1998cohen}).
\begin{cor}
    Let $A=\bigoplus _{n\geq 0}A_n$ be a standard graded $k$-algebra with 
    Hilbert series
    \[
    H(A,t)=\frac{h(t)}{(1-t)^d},
    \]
    such that $h$ is a polynomial of degree $s$ and has nonnegative coefficients. Then for $i\geq 0$ and all $r\geq s-i$ the numerator 
    polynomial of the Hilbert series of the Veronese submodule 
    $A\hri{r}{i}=\bigoplus _{n\geq 0}A_{rn+i}$ has only negative, real roots. If $r>s$ then these roots are distinct.
\end{cor}

\subsection{Quasi-polynomials}\label{subsection:quasipolys}
Our results naturally extend to quasi-polynomials. A function $f\colon \R \rightarrow \R$ is a quasi-polynomial if there are polynomials $f_0,\ldots, f_{l-1} \in \R[t]$ such that
\[
f(t) \ = \ f_i(t) \quad \text{ whenever } \quad t \equiv i \mod l.
\]
The degree of $f$ is defined as the maximum of the degrees of its constituents  $f_0,\ldots, f_{l-1} \in \R[t]$. The integer $l$ is called quasi-period and is not unique. For every quasi-polynomial $f$  of degree at most $d-1$ and quasi-period $l$ it can be shown
\[
\sum _{n\geq 0} f(n)t^n \ = \ \frac{h(t)}{(1-t^l)^d},
\]
where $h$ is a polynomial with $\deg h <ld$.

The following result applies to all quasi-polynomials whose numerator polynomial has only nonnegative coefficients.
\begin{prop}\label{cor:quasi}
Let $f\colon \R \rightarrow \R$ be a quasi-polynomial of degree $d-1$ and constituents $f_0,\ldots, f_{l-1} \in \R[t]$ and let 
\[
\sum _{n\geq 0} f(n)t^n \ = \ \frac{h(t)}{(1-t^l)^d},
\]
where $h\neq 0$ has only nonnegative coefficients. Then for all $0\leq i <l$ and all $r\geq 1$
\[
\sum _{n\geq 0} f_i(rn)t^n \ = \ \frac{U^d_{r,0}\left(h\hri{l}{i}\right)}{(1-t)^d}
\]
and $U^d_{r,0}\left(h\hri{l}{i}\right)$ has only nonpositive, real roots for $r\geq d$, or $U^d_{r,0}\left(h\hri{l}{i}\right)\equiv 0$ for all $r\geq 1$.
\end{prop}
\begin{proof}
As in the proof of Lemma \ref{lem:hecke}
\begin{align*}
(1-t)^d\sum_{n\geq 0}f_i(n)t^n&=(1-t)^d\sum_{n\geq 0}f(ln+i)t^n\\
&=\left((1-t^l)^d\sum_{n\geq 0}f(n)t^n\right)\hri{l}{i}\\
        &=h\hri{l}{i},
\end{align*}
which has nonnegative coefficients and $\deg h\hri{l}{i}\leq d$, or $h\hri{l}{i}=0$. Thus, the claim follows from Theorem~\ref{thm:maintheorem}.\end{proof}
The assumptions of Corollary \ref{cor:quasi} are, for example, satisfied for Ehrhart quasi-polynomials which count the number of lattice points in integer dilates of rational polytopes (i.e.\ polytopes with rational vertex coordinates) (see, e.g., \cite[Exercise~3.30]{BR}):
\begin{cor}\label{cor:Ehrhartquasi}
Let $P\subseteq \R^d$ be a $d$-dimensional rational polytope, let $l$ be a natural number such that $lP$ is a lattice polytope, and let $0\leq i<l$ be an integer. Then
\[
\sum _{n\geq 0}E((rln+i)P)t^n \ = \ \frac{h^{(r)}(t)}{(1-t)^{d+1}},
\]
where $h^{(r)}\in \R[t]$ has only nonpositive real roots for $r\geq d$, or $h^{(r)}\equiv 0$ for $r\geq 1$.
\end{cor}
\section{Optimality of bounds}\label{sec:optimality}

Let $h \in\R_{\geq 0}[t]$ be a polynomial with nonnegative coefficients. Theorem \ref{thm:maintheorem} shows that $U_{r,i}^d h$ has only real roots whenever $r\geq \deg h -i$, and all these roots are distinct if $r> \deg h$. It remains the question if these bounds on $r$ are optimal. More precisely, for all $s,i$ and $d$ we would like to determine the integers
\[
R(s,d,i):=\min _{l\in \N} \{U_{m,i}^d h \text{ real-rooted for all } h\in \R_{\geq 0}[t],\deg h =s , m\geq l\}
\]
and
\[
R^\circ (s,d,i):=\min _{l\in \N} \{U_{m,i}^d h \text{ has distinct real roots for all } h\in \R_{\geq 0}[t],\deg h =s, m \geq l\}.
\]

By construction it is clear that $R(s,d,i)$ and $R^\circ (s,d,i)$ are greater or equal to $i+1$. The following results show that the bounds on $r$ given in Theorem \ref{thm:maintheorem} are in most cases indeed optimal.

\begin{prop}
Let $i,s \geq 0$. Then $R^\circ (s,1,i) = \max \{s-i,i+1\}$, $R^\circ (s,2,s-1) = s$ and in all other cases
\[
 R^\circ (s,d,i) \ = \ \max \{s+1,i+1\}. 
\]
\end{prop}

\begin{proof}
Let $h=h_0+\cdots +h_s$ be a polynomial of degree $s$. By Theorem \ref{thm:maintheorem} $U_{r,i}^dh$ has distinct real roots if $r>s$. In order to show optimality, assume that $r=s>i$ and let $\alpha _i$ be a (negative) root of $a\hri{s}{i} _d$. It is not hard to see that a root exists if $d\geq 3$, and if $d=2$ and $i\leq s-2$. Let $h(t)=1-\tfrac{1}{\alpha _i}t^s$. Then by Lemma \ref{lem:linearexpr}

\[
U_{s,i}^d h \ = \ (1-\tfrac{1}{\alpha _i}t)a\hri{s}{i}
\]
which has a double root at $t=\alpha _i$. 

For $i=s-1$ and $d=2$ observe that $a\hri{s}{s-1}_2\equiv c$ is constant. Thus, if $h_j=0$ for all $0<j<s$, then $U_{s,s-1}^2 h = c(h_0+th_s)$. In the other case, if there is an $0<j<s$ with $h_j>0$ then $U_{s,s-1}^2 h$ has distinct zeros by Theorem \ref{thm:maintheorem}. Therefore $R^\circ (s,2,s-1)$ attains its minimal possible value $s$.

For $d=1$ observe that $a\hri{r}{i}_1\equiv 1$ for all $0\leq i<r$. Thus, by Lemma \ref{lem:linearexpr} $U_{r,i}^1 h$ is linear and therefore real-rooted whenever $r\geq s-i$.  Assume that $r=s-i-1$ and let $h(t)=h_0+t^sh_s$ with $h_0,h_s >0$. Then by Lemma \ref{lem:linearexpr}
\[
U_{s-i-1,i}^1 h \ = \ h_0+h_st^2
\]
which is not real-rooted.
\end{proof}

\begin{thm}\label{prop:optimality2}
Let $i,s \geq 0$ and $d\geq 1$. Then 
\[
 R (s,d,i) \ = \ \max \{s-i,i+1 \}.
\]
\end{thm}
\begin{proof}
Let $h$ be a polynomial of degree $s$. By Theorem \ref{thm:maintheorem} $U_{r_i}^dh$ has only real roots if $r\geq s-i$. To show optimality, we may assume that $s-i > i+1$. We consider polynomials of the form  $h(t)=h_0+h_st^s$ and $h_s \gg h_0 >0$. Then by Lemma \ref{lem:linearexpr}
\[
U_{s-i-1,i}^d h \ = \ h_0a_d\hri{s-i-1}{i}+h_st^2a_d\hri{s-i-1}{i+1}
\]
if $s-i-1>i+1$, and $U_{s-i-1,i}^d h \ = \ (h_0+h_st^2)a_d\hri{s-i-1}{i}$ if $s-i-1=i+1$. In this case it is clear that $U_{s-i-1,i}^d h$ has not only real roots.

We also want to see that in the case $s-i-1>i+1$. Let $\alpha<0$ be the largest (possibly non-existing) root of $a_d\hri{s-i-1}{i}$. Since $a_d\hri{s-i-1}{i+1} \prec a_d\hri{s-i-1}{i}$ and $h_0, h_s>0$ we have $U_{s-i-1,i}^d h(t)>0$ for all $t>\alpha$. Assume there exists a sequence $\{h_0^{(n)}\}_{n\in \N}$ of positive real numbers with $\lim _{n\rightarrow \infty}h_0^{(n)}=0$ such that $f^{(n)}:=h_0^{(n)}a_d\hri{s-i-1}{i}+h_st^2a_d\hri{s-i-1}{i+1}$ has only real roots. Then $\lim _{n\rightarrow \infty}f^{(n)}=h_st^2a_d\hri{s-i-1}{i+1}$, and for all $n\geq 0$ there are roots $\beta _n$ of $f^{(n)}$ with $\lim _{n\rightarrow \infty} \beta _n =0$. But this contradicts $f^{(n)}(t)>0$ for all $t>\alpha$. Therefore, for $0<h_0<<h_s$, that is, if $h_0$ is very small compared to $h_s$, $U_{s-i-1,i}^d h$ does not have only real roots.
\end{proof}

\section{Open questions}\label{sec:openquest}

The main question that remains is about the optimality of the bounds if we restrict to specific classes of formal power series, for example to Ehrhart series or Hilbert series of graded algebras. Since in these cases the coefficients of the numerator polynomial $h$ are natural numbers, finding the optimal bounds has a number theoretic flavor. For example, it is easy to see that $U_{s,0}^d h$ for binomial $h=h_0+h_st^s$ can only have a double root if $a\hri{s}{0}_d$ has a rational root.

For Ehrhart series, a result of Batyrev and Hofscheier \cite{batyrev2010generalization} comes in very useful. They characterized all lattice polytopes with $h^\ast$-polynomial of the form $1+h_st^s$. In particular, for $s\leq \frac{(d+1)}{2}$ there are lattice polytopes with  $h^\ast$-polynomial $1+h_st^s$ with arbitrarily large $h_s$. With the same arguments as in the proof of Proposition~\ref{prop:optimality2} this yields
\begin{cor}
Let $d\geq 1$, $i\geq 0$ and $s\leq \frac{(d+1)}{2}$ be integers. Then $U_{r,i}^dh^\ast (P)$ is real rooted for all $d$-dimensional lattice polytopes $P$ with $\deg h^\ast(P)=s$ whenever $r\geq \max \{s-i,i+1\}$, and the bound $\max \{s-i,i+1\}$ is optimal.
\end{cor}
It remains the question for the optimal bounds in the case  $d\geq s>\frac{(d+1)}{2}$.
\begin{quest}
For all $s$ and $d$ find the smallest number $N=N(s,d)$ such that for all lattice polytopes of dimension $d$ with $\deg h^\ast (P)=s$ the $h^\ast$-polynomial of $rP$ has only real roots for all $r\geq N$.
\end{quest}

\textbf{Acknowledgements.} We would like to thank Matthias Beck, Raman 
Sanyal for valuable feedback, and Christian Haase and Alan Stapledon for fruitful discussions.

\bibliographystyle{siam}
\bibliography{Veronese}

\end{document}